\documentclass[a4paper,final]{amsart}
\usepackage{amssymb}
\usepackage{graphicx}
\usepackage{color}
\usepackage{upref}
\usepackage{url}

\usepackage{calc}
\usepackage{ifthen}
\newcounter{hours}\newcounter{minutes}
\newcommand\printtime{%
 \setcounter{hours}{\time/60}%
 \setcounter{minutes}{\time-\value{hours}*60}%
\ifthenelse{\value{hours}<10}{0\thehours}{\thehours}
\ifthenelse{\value{minutes}<10}{:0\theminutes}{:\theminutes}}

\renewcommand{\phi}{\varphi}
\renewcommand{\epsilon}{\varepsilon}
\renewcommand{\theta}{\vartheta}

\newcommand{\C}[1]{\mathbf{C^{#1}}}
\newcommand{\reali}{{\mathbb{R}}}

\newcommand{\IV}{\rm IV}

\DeclareMathOperator{\argmin}{argmin}

\newtheorem{theorem}{Theorem}[section]
\newtheorem{proposition}[theorem]{Proposition}
\newtheorem{lemma}[theorem]{Lemma}
\newtheorem{definition}[theorem]{Definition}
\newtheorem{example}[theorem]{Example}

\numberwithin{equation}{section}
\allowdisplaybreaks
\begin{document}
\title[Braess Paradox]{On the Braess Paradox \\ with Nonlinear
  Dynamics and Control Theory}

\author[Colombo]{Rinaldo M.~Colombo} \address[Colombo]{\newline INDAM
  Unit, University of Brescia, Via Branze 38, I--25123 Brescia,
  Italy}
\email[]{\href{Rinaldo.Colombo@Ing.UniBs.It}{Rinaldo.Colombo@Ing.UniBs.It}}
\urladdr{\href{http://dm.ing.unibs.it/rinaldo/}{http://dm.ing.unibs.it/rinaldo/}}

\author[Holden]{Helge Holden} \address[Holden]{\newline Department of
  Mathematical Sciences, Norwegian University of Science and
  Technology, NO--7491 Trondheim, Norway,\newline {\rm and} \newline
  Centre of Mathematics for Applications,
  University of Oslo, P.O.\ Box 1053, Blindern, NO--0316 Oslo, Norway
} \email[]{\href{holden@math.ntnu.no}{holden@math.ntnu.no}}
\urladdr{\href{http://www.math.ntnu.no/~holden}{www.math.ntnu.no/\textasciitilde holden}}




\thanks{Partially supported by the Research Council of Norway and by
  the Fund for International Cooperation of the University of
  Brescia.}


\begin{abstract}
  We show the existence of the Braess paradox for a traffic network
  with nonlinear dynamics described by the Lighthill--Whitham-Richards
  model for traffic flow. Furthermore, we show how one can employ
  control theory to avoid the paradox. The paper offers a general
  framework applicable to time-independent, uncongested flow on
  networks. These ideas are illustrated through examples.
\end{abstract}

\maketitle

\section{Introduction}
\label{sec:intro}

Consider the following scenario: We have a simple network consisting
of two routes connecting $A$ to $B$, see Figure~\ref{fig:4Roads}.
\begin{figure}[h!]
  \centering
  \includegraphics[width = 0.6\textwidth]{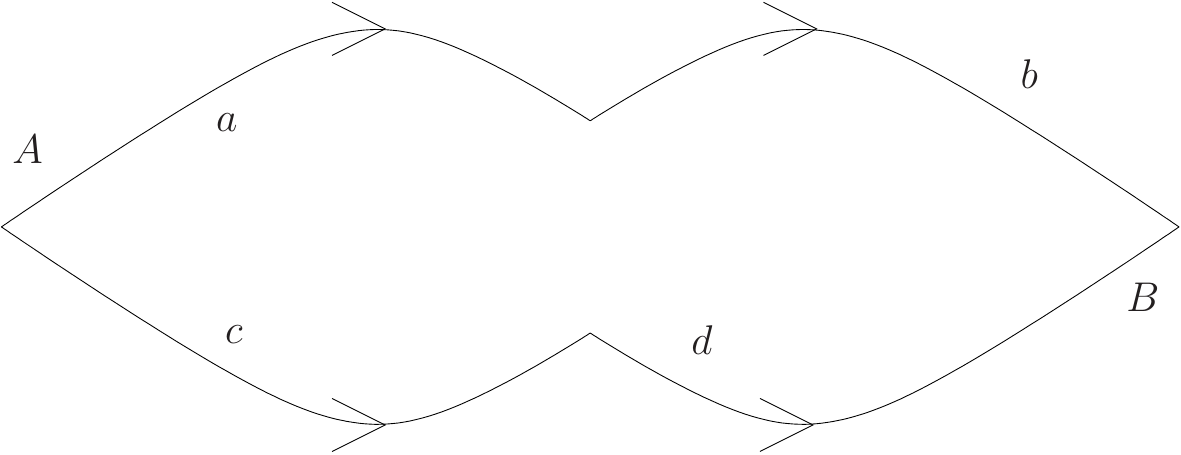}
  \caption{Network consisting of two routes connecting $A$ to $B$. The
    route $\alpha$ consists of the roads $a$ and $b$, the route
    $\beta$ consists of the roads $c$ and $d$.}
  \label{fig:4Roads}
\end{figure}
Each route consists of two roads. Roads $a$ and $d$ are identical, as
are roads $b$ and $c$. Traffic is unidirectional in the direction from
$A$ to $B$. Travel time along roads $a$ and $d$ are given by
$\rho/100$, where $\rho$ is the number of vehicles on that road, while
the travel time is $45$ for each of roads $b$ and $c$, irrespective of
the number of vehicles on that road.  In equilibrium, vehicles will
distribute evenly between the two routes connecting $A$ and $B$, i.e.,
roads $a$ \& $d$ and $b$ \& $c$.  Assuming that initially $m=4000$
vehicles start from $A$, we find a travel time of $65$ along each of
the two routes. Add a road $e$ as given in Figure~\ref{fig:5Roads},
and assume that the travel time is zero along this road.
\begin{figure}[h!]
  \centering
  \includegraphics[width = 0.6\textwidth]{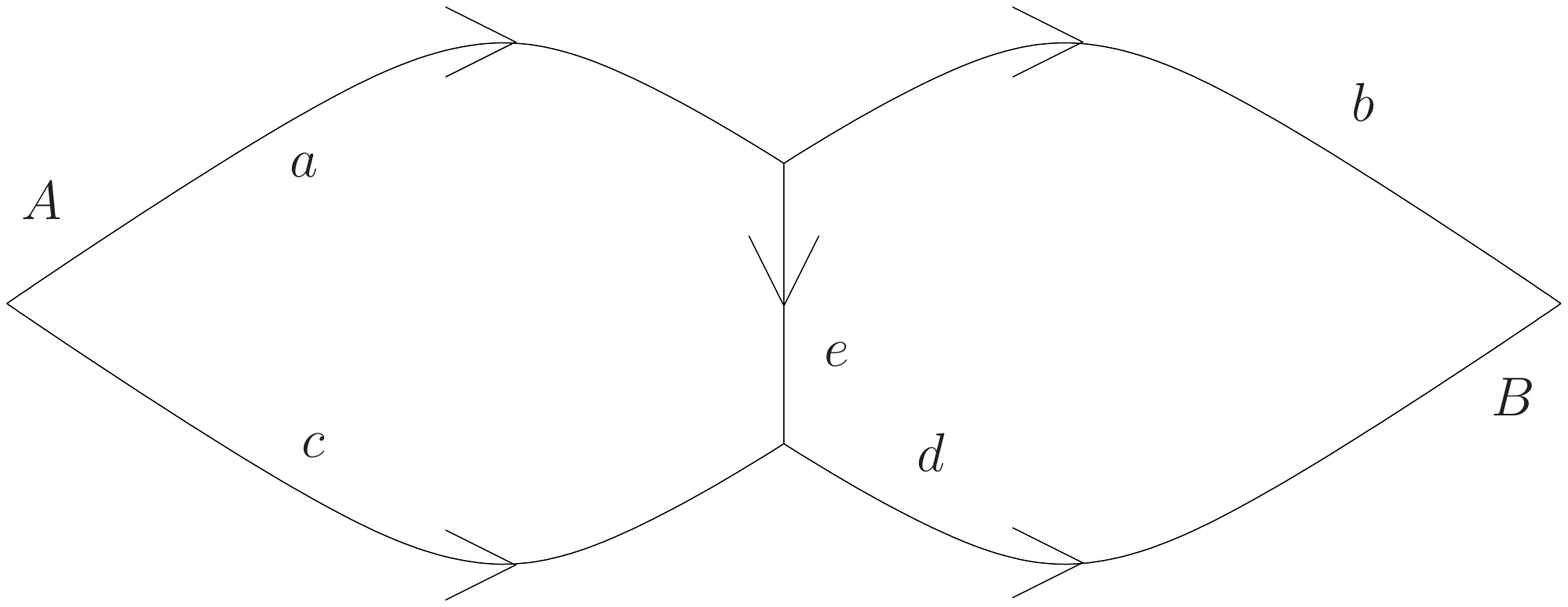}
  \caption{A network consisting of three routes $\alpha$, $\beta$, and
    $\gamma$ connecting $A$ to $B$. The route $\alpha$ comprises the
    roads $a$ and $b$, the route $\beta$ comprises the roads $c$ and
    $d$ and, finally, the route $\gamma$ consists of the roads $a$,
    $e$, and $d$.}
  \label{fig:5Roads}
\end{figure}
Drivers will start using the new road, reducing their travel time from
$65$ to $40$. However, as more and more drivers use the new road,
their travel time will increase to $80$. Now, no driver will have an
incentive to use the old roads, i.e., avoiding road $e$, as the travel
time along those roads will be $85$. Thus all drivers are worse off
than before, in spite of having a new road. This is the Braess paradox
in a nutshell: \textit{Adding a new road to a network may make travel
  times worse for all.} In both cases the equilibrium is a
\textit{Wardrop equilibrium} (i.e., all routes used have the same
travel time, and all unused routes have longer travel times) as well
as a \textit{Nash equilibrium}.


This is the simplest example of the Braess paradox, introduced (with a
different example) by Braess in 1968~\cite{BraessParadox}, see
also~\cite{NagurneyBoyce}. This example and some generalizations have
been studied in, e.g.,
\cite{Frank,HagstromAbrams,RoughgardenTardos}. In spite of the
unrealistic assumptions in the prevalent example above, the paradox
has turned out to be ubiquitous and intrinsic to dynamical networks.
The paradox also appears in other situations not modeling traffic
flow~\cite{SteinbergZangwill}, see, e.g.,~\cite{Pala} for an example
involving mesoscopic electron systems, and~\cite{CohenHorowitz} for an
example with mechanical springs. Furthermore, the paradox can be
reformulated in the context of game theory. In addition, there are
well documented examples of the paradox occurring in real-life traffic
situations, e.g., in Seoul~\cite{Baker} and
Stuttgart~\cite[pp. 57--59]{Knodel}, see also~\cite{YounGastnerJeong}.
Not surprisingly, the paradox has been well described also in general
media, see, e.g., \cite{Kolata, ArnottSmall, Vidal} and on Wikipedia
as well as YouTube. The extensive discussion about the Braess paradox
makes a complete reference list impossible, see, however,
\cite{Easley, Roughgarden,Roughgarden_paper}.  In this paper we only
refer to articles directly related to the research at hand.

Here we want to study the Braess paradox with a more realistic
nonlinear dynamics. More specifically, we want to model unidirectional
traffic along roads by a macroscopic model where only densities of
vehicles are considered. We believe this to be novel.  In this class
of models, introduced by Lighthill--Whitham~\cite{LighthillWhitham}
and Richards~\cite{Richards} (hereafter denoted the LWR model),
vehicles, described by a density $\rho$ rather than individually,
drive with a velocity determined by the density alone; higher density
yields slower speed while low density lets vehicles approach the speed
limit. At a maximum density with bumper-to-bumper vehicles, traffic
comes to a halt.  The dynamics is well described by the nonlinear
partial differential equation
\begin{equation}
  \label{eq:CL}
  \partial_t\rho+\partial_x\left(\rho \, v(\rho)\right)=0,
\end{equation}
see, e.g., \cite[pp. 11--18]{HoldenRisebro}.  The function
$q(\rho)=\rho v(\rho)$ is denoted the flux function, or, in the
context of traffic flow, the fundamental diagram. It is in general a
concave function that equals zero when $\rho$ vanishes and when $\rho$
equals the maximum possible road density.  Hyperbolic conservation
laws, as equations of the type~\eqref{eq:CL} are called, have been
used to study traffic on a network, starting with Holden and
Risebro~\cite{HoldenRisebro_network}, see, e.g., the book by Garavello
and Piccoli~\cite{GaravelloPiccoli}. Related results on a game
theoretic approach to network traffic through the LWR model,
see~\cite{BressanHan2012, BressanHan2013}. For general theory
concerning hyperbolic conservation laws we refer
to~\cite{HoldenRisebro}.

However, the Braess paradox describes an equilibrium situation, and it
is not relevant to include time variation. Rather, we want to study
stationary solutions where the velocity is a given function of the
density of vehicles on the road. At a junction, the differential
equation~\eqref{eq:CL} will in general, if the two roads have
different properties, establish a complicated wave pattern, creating
waves that emanate from the junction in both directions.  However, in
the equilibrium situation, this cannot happen, as it would create
time-dependent waves. Thus, we will set up the example in such a way
that no waves are created at junctions.

In this paper we analyze the same simple network as described above,
but with much more realistic dynamics. More general examples are of
course possible using the same methods. However, calculations become
more cumbersome and less transparent, and we here focus on presenting
the ideas of the model, exemplified on the simple network in
Figures~\ref{fig:4Roads} and~\ref{fig:5Roads}. For another approach to
the Braess paradox, see, e.g., \cite{DafermosNagurney}.

The prevalence of the Braess paradox is unwanted, and one would like
to take measures to prevent its occurrence. In the example in the
present paper, we use the velocity of the road $e$ as a control
parameter. By properly adjusting the speed limit on road $e$, one can
force the Braess paradox to disappear, and make the social optimum
coincide with the Nash equilibrium.

This can be illustrated in the simple example in the beginning of the
introduction. Given a ``benevolent dictator'' who wants to reduce the
total travel time and reach the social optimum, a short calculation
shows that, with $m=4000$, 1750 vehicles should follow each of the
routes $a$ \& $b$ and $c$ \& $d$, and the remaining $500$ vehicles
should follow the route $a$, $e$, and $d$.  Although a social optimum,
this situation is neither a Wardrop nor a Nash equilibrium.

This paper offers a framework applicable to general networks. The
input is, in addition to the network itself, the length and velocity
fields of each road as well as the influx.  We assume that traffic is
in the uncongested, or free, phase. This will prevent waves from
emanating from the junctions.

\section{A dynamic version of the Braess paradox}
\label{eq:Dyn}

\subsection{Notation and basic definitions}

Below, we denote $\reali^+ = [0, +\infty)$ and $S^n = \{\theta \in
[0,1]^{n} \mid \sum_j \theta_j \leq 1\}$ is the standard simplex in
$\reali^n$. The sphere centered at $\theta$ with radius $r$ is denoted
by $B_r (\theta)$.

Two points $A$ and $B$ are connected through a network of roads. Along
each road, traffic is described through the LWR
model~\eqref{eq:CL}. At each junction, the total flow exiting the
junction equals the incoming one, so that the total quantity of
vehicles is conserved.

The macroscopic description obtained solving~\eqref{eq:CL} along each
road also provides the full microscopic portrait of the
network. Indeed, once $\rho = \rho (t,x)$ is known along the road $r$
connecting, say, the junction at $A$ to that at $B$, the single
vehicle leaving from $A$ at time $t_o$ travels along $r$ according to
\begin{equation}
  \label{eq:ParticlePath}
  \left\{
    \begin{array}{l}
      \dot x = v\big(\rho(t, x (t))\big),
      \\
      x (t_o) = A \,.
    \end{array}
  \right.
\end{equation}
The travel time $\tau_r (t_o)$ along the road $a$ is then implicitly
defined by
\begin{equation}
  \label{eq:tau1}
  x \left(\tau_r (t_o)\right) = B \,.
\end{equation}
To compute $\tau_r (t_o)$, in general, one has first to
provide~\eqref{eq:CL} with initial and boundary data, then solve the
resulting initial-boundary value problem to obtain $\rho = \rho
(t,x)$, use this latter expression to solve the ordinary differential
equation~\eqref{eq:ParticlePath} and finally solve the
equation~\eqref{eq:tau1}. Observe that the right-hand side in the
ordinary differential equation in~\eqref{eq:ParticlePath} is in
general discontinuous, nevertheless in the present setting it is
well-posed, see~\cite{ColomboMarson}. In the present stationary
framework, this procedure can be pursued explicitly, as we detail
below in Example~\ref{ex:Simple}. Remark that, in a stationary regime,
all travel times are independent of the starting time $t_o$.

For the above travel times to be a reliable measure of the network
efficiency, it is necessary that they are independent from any
particular initial data. Also the standard initial-boundary value
problem for~\eqref{eq:CL} with zero initial density on the whole
network is unsatisfactory, since it would give results that depend on
the transient period necessary to fill the network. We are thus bound
to select \emph{stationary} solutions, assigning a constant inflow at
$A$ for all times $t \in \reali$. Moreover, to allow for stationary
solutions, we also assume that the total flow incoming at any junction
never exceeds the total capacity of the roads exiting that junction.

In the general LWR model~\eqref{eq:CL}, the flux function $q = q
(\rho)$ is a concave function that vanishes at zero density and at
$\rho_M$, the maximum density. The flux has a unique maximum for some
value $\rho_m \in (0,\rho_M)$. As usual, we refer to densities below
$\rho_m$ as the uncongested, or free, phase, and for densities above
$\rho_m$ as the congested phase. In the remaining part of the paper,
to obtain stationary solutions, we need to remain in the free phase
only, so that $\rho \in [0, \rho_m]$ throughout the network. In order
to simplify the notation we will use the normalization $\rho_m = 1$
for all roads. We will not make any assumptions on, or reference to,
$q$ above this value. Hence, on the flow function we pose the
following assumption:
\begin{description}
\item[(q)] $q \in \C3 ([0,1]; \reali^+)$, $q (0) = 0$, $q' > 0$ and
  $q'' \leq 0$.
\end{description}
\noindent Clearly, if $q$ satisfies~\textbf{(q)}, then the speed law
$v (\rho) = q (\rho) / \rho$ is well-defined, continuous, strictly
positive and weakly decreasing, see Lemma~\ref{lem:Speed}. As a
result, the travel along a road segment is a convex and increasing
function of the inflow.

\begin{lemma}
  \label{lem:TravelTime}
  Let $q$ satisfy~\textbf{(q)} with $q''' \leq 0$ and call $\phi = q
  (1)$. Then, the travel time $\tau (\theta)$, which is defined by $
  x\left(\tau (\theta)\right) = B$ where
  \begin{equation*}
    \mbox{ $x$ solves}  \quad
    \begin{cases}
      \dot x = v\left(\rho (t, x (t))\right),
      \\
      x (0) = A,
    \end{cases}
    \quad
    \mbox{ and }  \quad
    \mbox{$\rho$ solves}  \quad
    \begin{cases}
      \partial_t \rho + \partial_x q (\rho) = 0,
      \\
      q\left(\rho(t, A)\right) = \theta\phi,
    \end{cases}
  \end{equation*}
  is of class $\C2 ([0, 1]; \reali^+)$, weakly increasing and convex.
\end{lemma}

\noindent The proof follows directly from Lemma~\ref{lem:PropertiesV}.

\medskip

When $\gamma$ is a route consisting of the adjacent roads $r_1, r_2,
r_3, \ldots$, the travel time $\tau_\gamma (t_o)$ along $\gamma$ is
then defined as the sum $\sum_i \tau_{r_i}$ of the travel times of all
roads.

A network consists of several routes connecting $A$ to $B$. To
describe it, we enumerate each single road (or edge) and construct the
matrix $\Gamma$ setting
\begin{equation*}
  \Gamma_{ij} =
  \begin{cases}
    1 & \text{the road $r_i$ belongs to the route $\gamma_j$},
    \\
    0 & \text{otherwise.}
  \end{cases}
\end{equation*}
We now assign a constant total inflow $\phi$ at $A$ and call
$\theta_i$ the fraction of the drivers that reach $B$ along the route
$\gamma_j$.

A single road may well belong to more than one route, so that the flow
along the road $r_i$ is $\phi\Gamma_i \theta =\phi \sum_i \Gamma_{ij}
\theta_j$ and the travel time along that road results to be
$\tau_{r_i} (\Gamma_i \theta)$.  The total travel time $\tau_i$ along
the $i$th route is in general a function of all partition parameters,
more precisely
\begin{displaymath}
  \tau_{\gamma_j} (\theta)
  =
  \sum_i \Gamma_{ij} \; \tau_{r_i} \! ( \Gamma_i\theta).
\end{displaymath}

From a global point of view, it is natural to evaluate the quality of
a network through the mean global travel time\footnote{Also called
  \textit{average latency} of the system or \textit{social cost} of
  the network.} $T (\theta) = \sum_j \theta_j \, \tau_{\gamma_j}
(\theta)$ or, using matrix notation $\tau_r (\Gamma\theta) =
[\tau_{r_1} (\Gamma_1\theta)\, \cdots \, \tau_{r_n}
(\Gamma_n\theta)]$, we find
\begin{equation}
  \label{eq:MeanGlobalTravelTime}
  T (\theta)
  =
  \tau_r (\Gamma\theta) \; \Gamma \; \theta \,.
\end{equation}
We call \emph{globally optimal\footnote{Also called \textit{social
      optimum} for the system.}} a state $\theta_G \in S^n$ that
minimizes $T$ over $S^n$, i.e., $\theta_G = \argmin_{\theta \in S^n} T
(\theta)$. This \emph{social optimum} state conforms to Wardrop's
\emph{Second principle}, see~\cite[p.~345]{Wardrop}.

\begin{proposition}
  \label{prop:Convexity}
  Let all road travel times $\tau_{r_1}, \ldots, \tau_{r_m}$ be of
  class $\C2 ([0,1]; \reali^+)$, weakly increasing and convex. Then,
  the map $T$ is in $\C2 ([0,1];\reali^+)$ is convex.
\end{proposition}

\noindent The proof is deferred to the Appendix.

For brevity, we call \emph{relevant} those travel times $\tau_i$ such
that $\theta_i \neq 0$.

\begin{definition}
  \label{def:Eq}
  A state $\bar \theta \in S^n$ is an \emph{equilibrium state} if all
  relevant travel times coincide, i.e., for all $i,j \in \{1, \ldots,
  n\}$
  \begin{displaymath}
    \mbox{ if } \bar \theta_i \neq 0 \mbox{ and } \bar \theta_j \neq 0,
    \mbox{ then } \tau_i(\bar\theta) = \tau_j(\bar\theta)  = \bar \tau\,,
  \end{displaymath}
  the common value $\bar \tau$ of the travel times being the
  \emph{equilibrium time}.
\end{definition}
In other words, at equilibrium all drivers need the same time to go
from $A$ to $B$. A common criterion for optimality goes back to
Pareto.

\begin{definition}
  \label{def:Pareto}
  An equilibrium state $\theta^P \in S^n$ is a \emph{local Pareto
    point} if there exists a positive $\delta$ such that for all
  $\theta \in B_\delta (\theta^P) \cap S^n$ if there exists a $j$ such
  that $\tau_{\gamma_j} (\theta) < \tau_{\gamma_j} (\theta^P)$, then
  there exists also a $k$ such that $\tau_{\gamma_k} (\theta) >
  \tau_{\gamma_k} (\theta^P)$.
\end{definition}

In other words, no (small) perturbation of a Pareto point may reduce
all travel times.

However, from a \emph{``selfish''} point of view, each driver aims at
reducing his/her own travel time.  It is then natural to introduce the
following definition.

\begin{definition}
  \label{def:Nash}
  An equilibrium state $\theta^N \in S^n$ is a \emph{local Nash point}
  if there exists a positive $\delta$ such that for all $\epsilon \in
  (0, \delta]$ and all $j,k = 1, \ldots, n$,
  \begin{displaymath}
    \mbox{if } \theta^N + \epsilon e_j - \epsilon e_k\in S^n
    \mbox{, then }
    \tau_{\gamma_j} (\theta^N + \epsilon e_j - \epsilon e_k)
    >
    \tau_{\gamma_k} (\theta^N) \,,
  \end{displaymath}
  where $e_j$ is the unit vector directed along the $j$th axis.
\end{definition}

In other words, it is not convenient for $\epsilon$ drivers to change
from route $k$ to route $j$, for any $j,k =1,\dots, n$.

\medskip

\begin{example}
  \label{ex:Simple}
  Consider the simple case of the network in Figure~\ref{fig:Simple},
  and assume that its dynamics is described as follows:

  \smallskip\noindent
  \begin{tabular}{@{}ccccc@{}}
    Road & Length & Density & Model & Flow
    \\[2pt]
    $a$ & $3/2$ & $\rho$ & $\partial_t \rho + \partial_x \left(\rho \,
      v (\rho)\right) = 0$ & $q (\rho) =
    \left(-1+\sqrt{1+8\rho}\right)/4$
    \\[2pt]
    $b$ & $1$ & $R$ & $\partial_t R + \partial_x \left(R \, V
      (R)\right) = 0$ & $Q (R) = -1 + \sqrt{1+R}$
  \end{tabular}
  \smallskip

  The maximal inflow $\phi$ at $A$ that, for any $\theta \in [0,1]$,
  can be partitioned in $\theta\,\phi$ along $a$ and $(1-\theta)\phi$
  along $b$ is $\min\left\{q (1),Q (1)\right\} = \sqrt{2}-1$.
  \begin{figure}[h!]
    \centering
    \includegraphics[width = 0.6\textwidth]{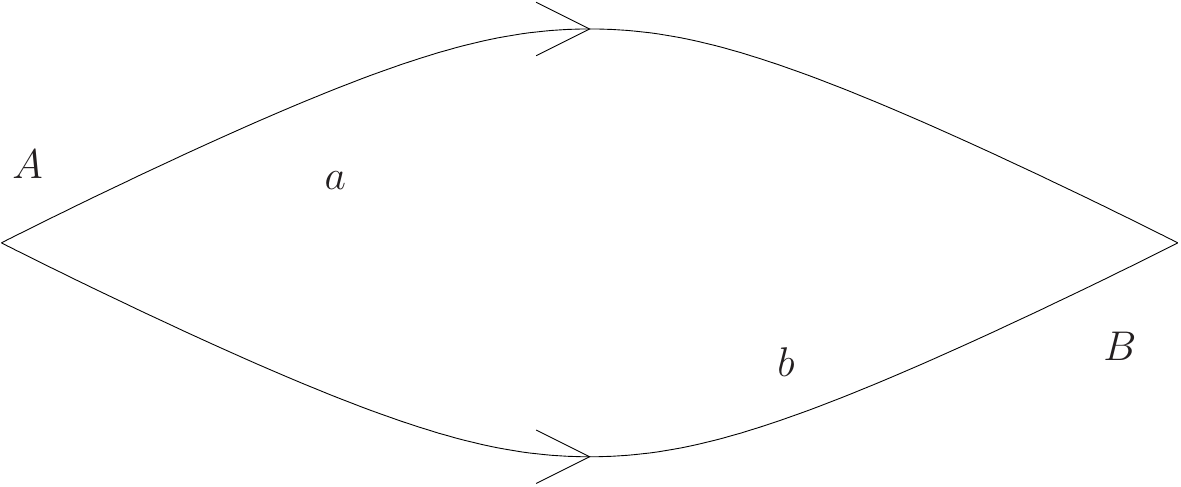}
    \caption{A simple network connecting $A$ to $B$ where the globally
      optimal state differs from the Nash optimal one.}
    \label{fig:Simple}
  \end{figure}
  With this constant inflow as left boundary data in~\eqref{eq:CL},
  the resulting (stationary) densities are
  \begin{displaymath}
    \text{$\rho = (1 + 2 \, \theta\, \phi) \, \theta \, \phi$
      along road $a$, and
      $R = \left(2 + (1-\theta) \phi \right) (1-\theta) \, \phi$
      along road $b$} \,.
  \end{displaymath}
  The corresponding constant traffic speeds
  \begin{displaymath}
    \text{$v (\rho) = (1+2\, \theta \, \phi) ^{-1}$ along road $a$,
      and
      $V (R) = (2 + (1-\theta) \phi)^{-1}$ along road $b$},
  \end{displaymath}
  inserted in~\eqref{eq:ParticlePath}, lead to the following travel
  times on the two roads:
  \begin{displaymath}
    \text{$\tau_a (\theta)
      =
      3(1 + 2\,\theta\,\phi)/2$ along road $a$, and
      $\tau_b (1-\theta)
      =
      2 + (1-\theta) \phi$ along road $b$} \,.
  \end{displaymath}
  Finally, the mean global travel time defined
  at~\eqref{eq:MeanGlobalTravelTime} is
  \begin{displaymath}
    T (\theta)
    =
    2 + \phi - \frac{1+4\,\phi}{2}\, \theta + 4 \, \theta^2 \, \phi\,.
  \end{displaymath}
  According to Definition~\ref{def:Nash}, we have a unique Nash point
  at $\theta^N$ and a unique globally optimal state at $\theta_G$,
  where
  \begin{displaymath}
    \theta^N =
    \begin{cases}
      0,  & \phi \in [0, 1/6),     \\
      \frac{1+2\, \phi}{8\, \phi}, & \phi \in [1/6, \sqrt{2}-1],
    \end{cases}
    \qquad
    \theta_G =
    \begin{cases}
      0, & \phi  \in  [0, 1/12),     \\
      \frac{1+4\, \phi}{16\, \phi}, & \phi \in [1/12, \sqrt{2}-1].
    \end{cases}
  \end{displaymath}
  Clearly, $\theta^N$ is also a Pareto point according to
  Definition~\ref{def:Pareto}. Note that the globally optimal state
  may well differ from the Nash optimal one and both depend on the
  total inflow $\phi$, see Figure~\ref{fig:SimpleNum}.
  \begin{figure}[h!]
    \centering
    \includegraphics[width=0.75\textwidth]{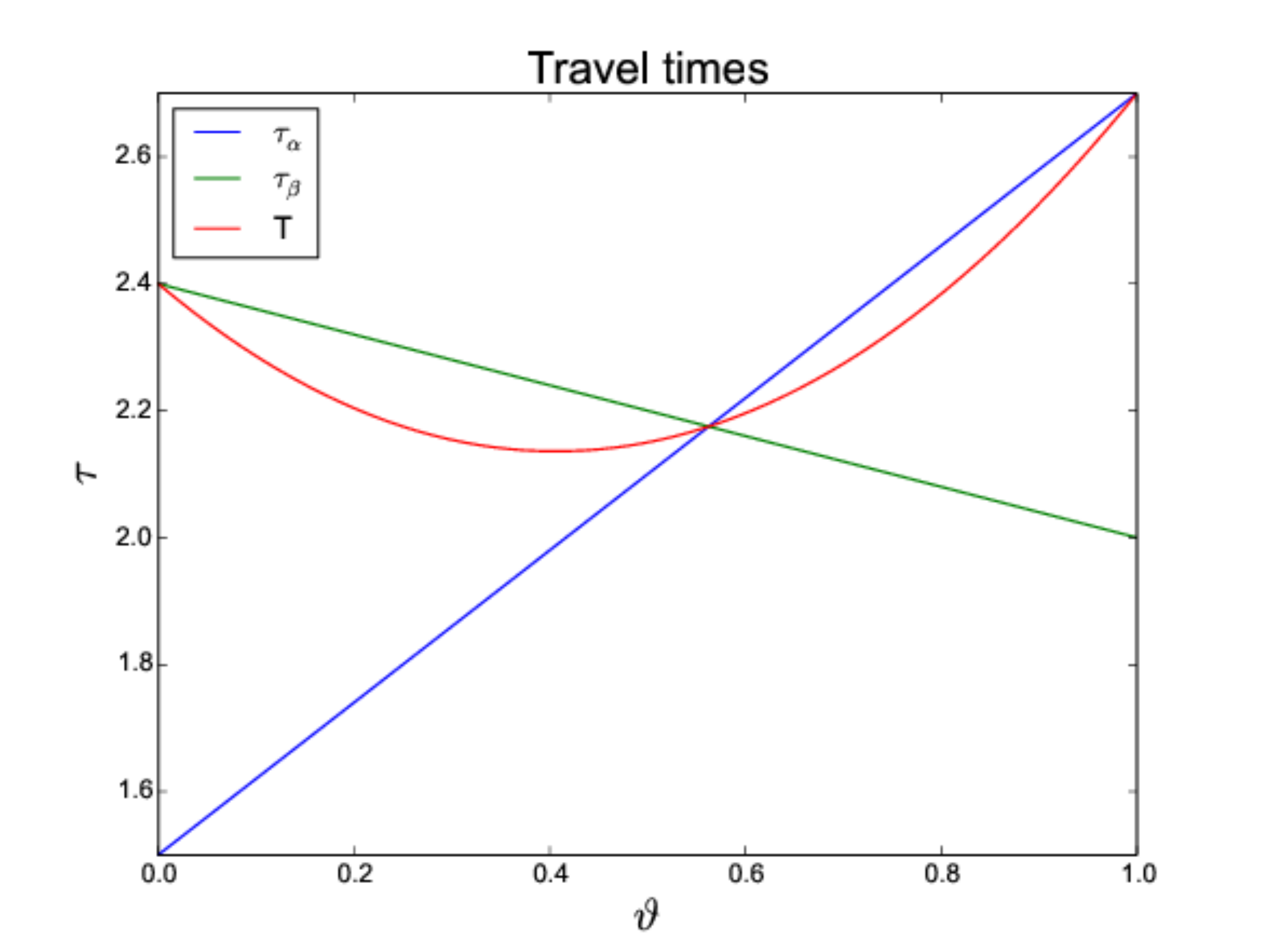}  
    \caption{Travel times of the situation described in
      Example~\ref{ex:Simple} with $\phi = 0.4$, so that $\theta^N =
      0.5625$ and $\theta_G = 0.40625$.}
    \label{fig:SimpleNum}
  \end{figure}
\end{example}

\subsection{The case of four roads}

Consider the network in Figure~\ref{fig:4Roads}.  The network is given
by two routes, denoted $\alpha$ and $\beta$, connecting $A$ and
$B$. The route $\alpha$ consists of roads $a$ and $b$, the route
$\beta$ consists of roads $c$ and $d$. Roads $a$ and $d$ have the same
length $\ell$ and the same fundamental diagram $q$. Similarly, roads
$b$ and $c$ share the same length $L$ and the same flow density
relation.  Traffic is always assumed to be unidirectional from $A$ to
$B$, and no obstructions, e.g., traffic lights, are encountered at the
junctions.

Along each road, the dynamics of traffic is described by the LWR
model~\eqref{eq:CL} with flux functions that lead to the travel times
\begin{displaymath}
  \tau_a (\theta) = \tau_d (\theta)
  \quad \mbox{ and } \quad
  \tau_b (\theta) = \tau_d (\theta) \,,
\end{displaymath}
so that the travel time $\tau_\alpha (\theta)$ along the route
$\alpha$ and $\tau_\beta (1-\theta)$ along the route $\beta$, are
\begin{displaymath}
  \tau_\alpha (\theta) = \tau_a (\theta) + \tau_b (\theta)
  \quad \mbox{ and } \quad
  \tau_\beta (1-\theta) = \tau_a (1-\theta) + \tau_b (1-\theta) \,.
\end{displaymath}
Then, $\theta \mapsto \tau_\alpha (\theta)$ is (weakly) increasing,
while $\theta \mapsto \tau_\beta (1-\theta)$ is (weakly)
decreasing. Since $\tau_\alpha (1/2) = \tau_\beta (1/2)$, we have that
$\theta^N = 1/2$ is a Nash (and also Pareto) point for this system. It
is easy to verify that $(\theta^N,\theta^N)$ is also globally optimal,
since it is the argument that minimizes $T (\theta_1,\theta_2)$ over
the simplex $S^2$.

\subsection{The case of five roads}

We now introduce a new road in Figure~\ref{fig:4Roads}, passing to the
network described in Figure~\ref{fig:5Roads}.  The new road $e$, which
has the direction from $a$ to $d$, has length $\tilde \ell$ and its
dynamics is characterized by a flow function $\tilde q$
satisfying~\textbf{(q)}.  The presence of the road $e$ allows us to
consider the route $\gamma$ connecting $A$ to $B$ consisting of the
roads $a$, $e$, and $d$. For all $\theta_1, \theta_2 \in [0,1]$ such
that $\theta_1 + \theta_2 \leq 1$, we now let the inflow $\theta_1 \,
\phi$ enter $\alpha$, $\theta_2 \, \phi$ enter $\beta$ and the
remaining $(1-\theta_1-\theta_2) \, \phi$ enter $\gamma$. The travel
times along the three routes are then:
\begin{equation}
  \label{eq:tau}
  \begin{aligned}
    \tau_\alpha (\theta_1, \theta_2) & = \tau_a (1-\theta_2) + \tau_b
    (\theta_1) ,
    \\
    \tau_\beta (\theta_1, \theta_2) & = \tau_b (\theta_2) + \tau_a
    (1-\theta_1),
    \\
    \tau_\gamma (\theta_1, \theta_2) & = \tau_a (1-\theta_2) + \tau_e
    (1-\theta_1-\theta_2) + \tau_a (1-\theta_1) \,.
  \end{aligned}
\end{equation}
Observe that $\tau_\alpha (\theta, \theta) = \tau_\beta (\theta,
\theta)$.

The mean global travel time is
\begin{equation}
  \label{eq:T}
  T (\theta_1,\theta_2)
  =
  \theta_1 \, \tau_\alpha (\theta_1, \theta_2)
  +
  \theta_2 \, \tau_\beta (\theta_1, \theta_2)
  +
  (1-\theta_1-\theta_2) \, \tau_\gamma (\theta_1, \theta_2) \,.
\end{equation}

\subsection{The Braess paradox}
\label{subs:Braess}

We now compare the travel times obtained in the two cases described by
Figures~\ref{fig:4Roads} and~\ref{fig:5Roads}. To this end, observe
that the travel times $\tau_\alpha^{\IV}$ and $\tau_\beta^{\IV}$ in
the case of four roads, and referring to Figure~\ref{fig:4Roads}, are
obtained from those in the $5$ roads case setting
\begin{displaymath}
  \tau_\alpha^{\IV} (\theta) = \tau_\alpha (\theta, 1-\theta)
  \quad \mbox{ and } \quad
  \tau_\beta^{\IV} (\theta) = \tau_\beta (\theta, 1-\theta).
\end{displaymath}

\begin{theorem}
  \label{thm:4roads}
  Let the travel times $\tau_a, \tau_b, \tau_e \in \C0 ([0,1];
  \reali^+)$ be non decreasing and assume that $\tau_a$ or $\tau_b$
  are not constant. If the travel times defined in~\eqref{eq:tau}
  satisfy
  \begin{equation}
    \label{eq:Braess}
    \tau_\alpha (1/2, 1/2)
    <
    \tau_\gamma (0,0)
    <
    \tau_\alpha (0,0),
  \end{equation}
  then:
  \begin{itemize}
  \item $\theta^N \equiv (0,0)$ is the unique local Nash point for the
    network with five roads in Figure~\ref{fig:5Roads};
  \item the corresponding equilibrium time $\tau_\gamma (0,0)$ is
    worse than the globally optimal configuration for the network with
    four roads in Figure~\ref{fig:4Roads}.
  \end{itemize}
  Under the above conditions we have the occurrence of the Braess
  paradox.
\end{theorem}

Observe that the point $\theta^P \equiv (1/2, 1/2)$ is the unique
Pareto point for the five roads networks.

Condition~\eqref{eq:Braess} allows us to construct several examples
illustrating the Braess paradox.

\begin{example}
  \label{ex:Braess}
  With the notation in Figure~\ref{fig:5Roads}, choose
  \begin{displaymath}
    \begin{array}{cccr@{\,}c@{\,}lr@{\,}c@{\,}l}
      \mbox{Road} & \mbox{Length} & \mbox{Density} & \mbox{Flow} &
      \\
      a,\, d & 1 & \rho & q (\rho) & = & \ln (1+\rho)
      \\
      b, \, c & 1 & R & Q (R) & = & R \, V
      & (V & \in & \reali)
      \\
      e & 1 & \tilde\rho & \tilde q (\tilde \rho) & = & \tilde \rho \, \tilde v
      & (\tilde v & \in & \reali)
    \end{array}
  \end{displaymath}
  Condition~\eqref{eq:Braess} then becomes
  \begin{displaymath}
    \frac{e^\phi -1}{\phi}
    <
    \frac{1}{V} - \frac{1}{\tilde v}
    <
    \frac{2}{\phi}(e^\phi-e^{\phi/2}),
  \end{displaymath}
  and, for any $\phi \in \left(0, \min\{\ln2, V, \tilde v\} \right]$,
  it can easily be met for suitable $V$, $\tilde v$, see
  Figure~\ref{fig:ex:Braess}.
  \begin{figure}[!h]
    \centering
    \includegraphics[width=0.49\textwidth, trim=90 0 35
    0]{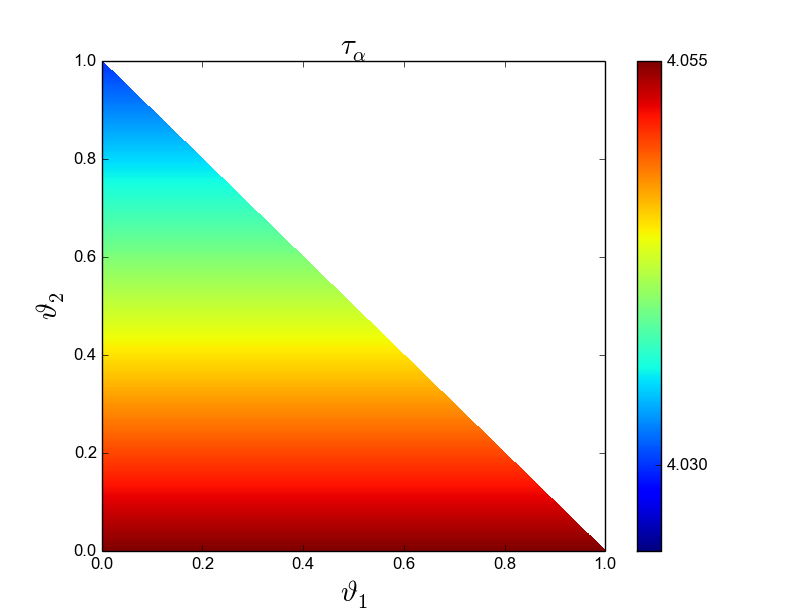}%
    \includegraphics[width=0.49\textwidth, trim=35 0 90 0]{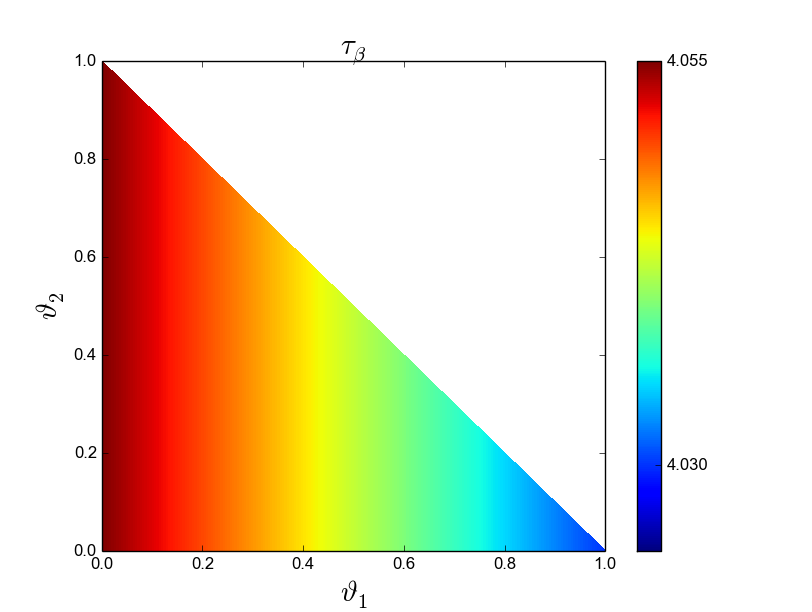}\\
    \includegraphics[width=0.49\textwidth, trim=90 0 35
    0]{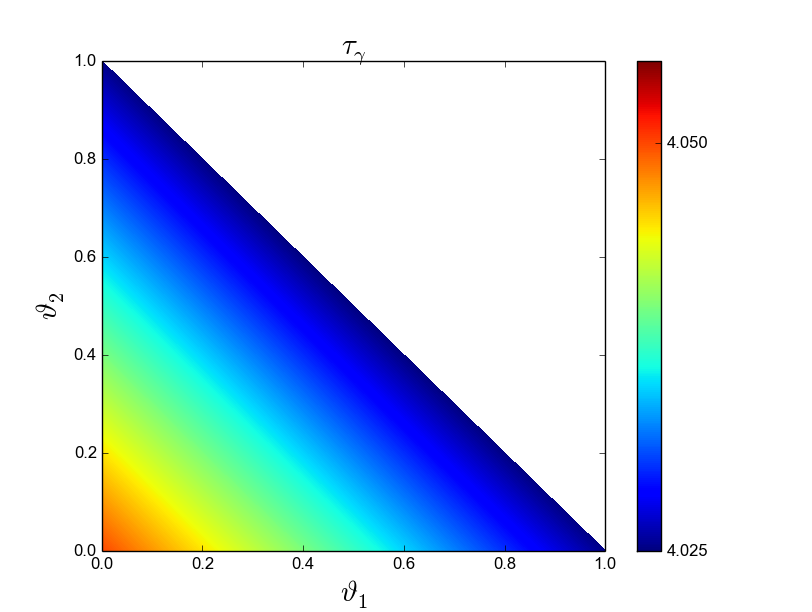}%
    \includegraphics[width=0.49\textwidth, trim=35 0 90 0]{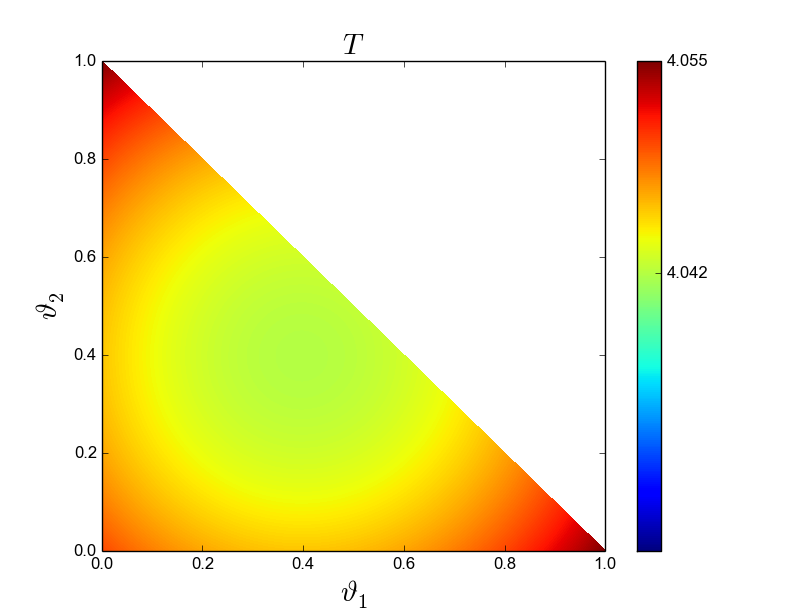}\\
    \caption{Contour plots of the travel times related to
      Example~\ref{ex:Braess} with $V=0.33$, $\tilde v=0.5$, $\ell = L
      = \tilde \ell = 1$, $\phi = 0.05$. Above, $\tau_\alpha$ and
      $\tau_\beta$; below $\tau_\gamma$ and the global travel time
      $T$. The color scales to the right are the same in all figures
      and display the maximal and minimal values of the diagrams to
      their left.}
    \label{fig:ex:Braess}
  \end{figure}
\end{example}

\section{Control theory for the novel road --- or how to cope with the
  Braess paradox}
\label{sec:control}

Our next aim is proving that in the case of the network in
Figure~\ref{fig:5Roads}, a carefully chosen speed limit imposed on the
novel road $\gamma$ makes the Nash optimal state coincide with the
globally optimal one.

We use the same notation as in Section~\ref{subs:Braess}, but we use
the travel time $\tilde \tau$ along the $e$ road as control
parameter. Equivalently, we impose that the speed along the road
$\gamma$ is $\tilde v$, so that
\begin{equation}
  \label{eq:taue}
  \tau_e (\theta_1,\theta_2) = \tilde \tau.
\end{equation}
The next theorem says that there exists an optimal control.
\begin{theorem}
  \label{thm:Control}
  Let the travel time $\tau_a,\tau_b \in \C0 ([0,1]; \reali^+)$ be non
  decreasing and convex, one of the two being strictly convex. Then,
  there exists a constant travel time $\tilde \tau \in \reali^+$ such
  that the network in Figure~\ref{fig:5Roads} admits a partition
  $(\theta_*, \theta_*)$ which is a Nash optimal state and also
  globally minimizes the mean global travel time.
\end{theorem}
Thus, by carefully selecting the travel time, or, equivalently,
adjusting the maximum speed, one can avoid the occurrence of the
Braess paradox. Moreover, the Nash equilibrium is steered to become
globally optimal.

\appendix
\section{Technical details}

\begin{lemma}
  \label{lem:Speed}
  Let $q$ satisfy~\textbf{(q)}. Then, the speed $v = v (\rho)$ defined
  by
  \begin{displaymath}
    v (\rho)
    =
    \left\{
      \begin{array}{lr@{\;}c@{\;}l}
        q' (0) & \rho & = & 0
        \\
        q (\rho) / \rho & \rho & > & 0
      \end{array}
    \right.
  \end{displaymath}
  is well-defined, continuous in $[0, \rho_m]$, strictly positive and
  weakly decreasing.
\end{lemma}

\begin{proof}
  Continuity follows from l'H\^opital's rule. By straightforward
  computation we find
  \begin{displaymath}
    v' (\rho)
    =
    \begin{cases}
      \frac{\rho \, q' (\rho) - q (\rho)}{\rho^2} & \rho > 0\,,
      \\
      \frac{1}{2} \, q'' (0) & \rho = 0\,,
    \end{cases}
    \qquad
    v'' (\rho)
    =
    \begin{cases}
      \frac{q'' (\rho)}{\rho} - 2 \frac{q' (\rho)}{\rho^2} + 2 \frac{q
        (\rho)}{\rho^3} & \rho > 0\,,
      \\
      \frac{1}{3} q''' (0) & \rho = 0 \,.
    \end{cases}
  \end{displaymath}
  By the concavity of $q$, we have $q' (0) \geq q (\rho)/\rho \geq q'
  (\rho)$, implying that $v' \leq 0$.
\end{proof}

\begin{lemma}
  \label{lem:PropertiesRho}
  Let $q$ satisfy~\textbf{(q)}. Then, the map $\rho \colon \theta
  \mapsto \rho (\theta)$ defined by
  \begin{displaymath}
    q\left(\rho (\theta)\right) = \theta\phi
  \end{displaymath}
  satisfies:
  \begin{enumerate}
  \item $\rho \in \C2 ([0,1]; [0,1])$ and $\rho (0) = 0$;
  \item $\rho' (\theta) > 0$ and $\rho'' (\theta) > 0$ for all $\theta
    \in [0,1]$;
  \item if $q$ is strictly convex, then $\rho'' (\theta) > 0$ for all
    $\theta \in [0,1]$.
  \end{enumerate}
\end{lemma}

\begin{proof}
  Existence and regularity of $\rho$ are immediate. Moreover,
  by~\textbf{(q)} and $q (\rho (\theta)) = \theta \, \phi$, it follows
  that
  \begin{displaymath}
    \rho (0) = 0
    \,,\qquad
    \rho' (\theta) = \frac{\phi}{q'\left(\rho (\theta)\right)}  > 0
    \quad \mbox{ and } \quad
    \rho'' (\theta)
    =
    -
    \frac{\phi^2 \, q''\left(\rho (\theta)\right)}{\left(q'\left(\rho (\theta)\right)\right)^3} \geq 0\,,
  \end{displaymath}
  and the latter inequality is strict as soon as $q$ is strictly
  convex.
\end{proof}

\begin{lemma}
  \label{lem:PropertiesV}
  Let $q$ satisfy~\textbf{(q)}. Then, the map $\theta \mapsto
  1/v\left(\rho (\theta)\right)$ is weakly increasing. If, moreover,
  $q''' (\rho) \leq 0$ for all $\rho \in [0,1]$, then the map $\theta
  \mapsto 1/v\left(\rho (\theta)\right)$ is convex.
\end{lemma}

\begin{proof}
  We find
  \begin{displaymath}
    \frac{d}{d\theta} \left(\frac{1}{v\left(\rho (\theta)\right)}\right)
    =
    - \frac{v'\left(\rho (\theta)\right) \, \rho' (\theta)}{\left(v\left(\rho (\theta)\right)\right)^2}
    \geq
    0 \,.
  \end{displaymath}
  Moreover, using the explicit expressions above,
  \begin{align*}
    \frac{d}{d\theta} \left(\frac{1}{v\left(\rho
          (\theta)\right)}\right) & = - \frac{v'\left(\rho
        (\theta)\right) \, \rho' (\theta)}{\left(v\left(\rho
          (\theta)\right)\right)^2}
    \\
    & = -\frac{ \frac{\rho (\theta) \, q'\left(\rho (\theta)\right) -
        q\left(\rho (\theta)\right)}{\left(\rho (\theta)\right)^2} \;
      \frac{\phi}{q'\left(\rho
          (\theta)\right)}}{\frac{\left(q\left(\rho
            (\theta)\right)\right)^2}{\left(\rho (\theta)\right)^2}}
    \\
    & = \left( \frac{1}{q\left(\rho (\theta)\right) \, q'\left(\rho
          (\theta)\right)} - \frac{\rho (\theta)}{\left(q\left(\rho
            (\theta)\right)\right)^2} \right) \phi,
    \\[5mm]
    \frac{d^2}{d\theta^2} \left(\frac{1}{v\left(\rho
          (\theta)\right)}\right) & = - 2 \, \frac{\rho' (\theta) \,
      \phi}{\left(q\left(\rho (\theta)\right)\right)^3}
    \\
    & \quad \times \left[ \frac{1}{2} \left(\frac{q\left(\rho
            (\theta)\right)}{q'\left(\rho (\theta)\right)}\right)^2
      q''\left(\rho (\theta)\right) + q\left(\rho (\theta)\right) -
      \rho (\theta) \, q'\left(\rho (\theta)\right) \right].
  \end{align*}
  Call $f (\rho) = \frac{1}{2} \left(\frac{q(\rho)}{q'(\rho)}\right)^2
  q''(\rho) + q(\rho) - \rho \, q'(\rho)$. Observe that $f (0) = 0$
  and
  \begin{displaymath}
    f' (\rho)
    =
    \frac{1}{2}
    \left(\frac{q(\rho)}{q'(\rho)}\right)^2
    q'''(\rho)
    +
    \frac{\left(q (\rho) - \rho \, q' (\rho)\right)q'' (\rho)}{q' (\rho)}
    -
    \frac{q (\rho) \, \left(q'' (\rho)\right)^2}{\left(q' (\rho)\right)^3}
    \leq
    0,
  \end{displaymath}
  thereby completing the proof.
\end{proof}
The assumption that $q''' (\rho) \leq 0$ is sufficient, but not
necessary, to obtain convexity of the travel time.

\begin{proof}[Proof of Proposition~\ref{prop:Convexity}.]
  Observe that if $f \in \C2 (\reali^+; \reali)$ is convex and
  increasing, then also the map $x \mapsto x\,f (x)$ is convex and
  increasing. By Lemma~\ref{lem:PropertiesRho}, for all $i=1,
  \ldots,m$, the map $\xi \mapsto \tau_{r_i} (\xi) \, \xi$ is convex
  for $\xi \in [0,1]$. Hence, also the map $\theta \mapsto \sum_i
  \tau_{r_i} (\theta_i) \, \theta_i$ is convex for $\theta \in
  [0,1]^n$. Since $\Gamma_{ij} \in \{0,1\}$, also the map $\theta
  \mapsto T (\theta)$ is convex.
\end{proof}

\begin{proof}[Proof of Theorem~\ref{thm:4roads}]
  By Definition~\ref{def:Nash}, the configuration $\theta^N$ with
  $\theta^N_1 = \theta^N_2 = 0$ is clearly an equilibrium, the only
  relevant time being the equilibrium
  \begin{displaymath}
    \bar \tau
    =
    \tau_\gamma (0,0)
    =
    2 \tau_a (1) + \tau_e (1)
    =
    2 \, \frac{\ell}{v\left(\rho (1)\right)}
    +
    \frac{\tilde \ell}{\tilde v\left(\tilde\rho (1)\right)} \,.
  \end{displaymath}
  By~\eqref{eq:Braess}, it is also a Nash point, since $\tau_a (0,0) =
  \tau_\beta (0,0) > \bar \tau$ and, by continuity, the same
  inequality holds in a neighborhood of $\theta^N$.

  Assume there exists an other equilibrium point $\bar \theta$ in the
  interior of $S^2$. Then, by symmetry, $\bar\theta_1 = \bar\theta_2$
  and, by Definition~\ref{def:Nash},
  \begin{equation}
    \label{eq:i}
    \tau_b (\bar\theta_1) - \tau_a (1-\bar\theta_1)
    =
    \tau_e (1-2\bar\theta_1) \,.
  \end{equation}
  By assumption, the left-hand side above is a strictly increasing
  function of $\theta_1$, while the right-hand side is weakly
  decreasing, so that
  \begin{align*}
    \tau_e (1-2 \bar\theta_1) & \leq \tau_e (1)
    \\
    & < \tau_b (0) + \tau_a (0) - 2 \tau_a (1) \qquad
    \mbox{by~\eqref{eq:Braess}}
    \\
    & \leq \tau_b (0) + \tau_a (0) - 2 \tau_a (0)
    \\
    & \leq \tau_b (0) - \tau_a (0)
    \\
    & \leq \tau_b (\bar\theta_1) - \tau_a (1-\bar\theta_1),
  \end{align*}
  which contradicts~\eqref{eq:i}.  To complete the proof of the
  uniqueness of the Nash points, consider the configuration $(0,1)$.
  In this case, the only relevant time is $\tau_\alpha (0,1)$ and
  \begin{displaymath}
    \tau_\alpha (1,0)
    =
    \tau_a (1) + \tau_b (1)
    >
    \tau_a (0) + \tau_b (0)
    =
    \tau_\beta (1,1),
  \end{displaymath}
  proving that $(1,0)$ is not a Nash point. The case of $(0,1)$ is
  entirely analogous.

  Finally, observe that the globally optimal time for the case of four
  roads is $\tau_\alpha (1/2,1/2) = \tau_b (1/2, 1/2)$ and the
  leftmost bound in~\eqref{eq:Braess} allows to complete the proof.
\end{proof}

\begin{lemma}
  \label{lem:Theta}
  Let the travel time $\tau_a,\tau_b \in \C0 ([0,1]; \reali^+)$ be non
  decreasing and convex, at least one of the two being strictly
  convex. Then, there exists a map $\Theta \in \C0 (\reali^+;[0,
  1/2])$ such that the partition $\left(\Theta (\theta), \Theta
    (\theta)\right)$ is the point of global minimum of the mean travel
  time $T$ defined in~\eqref{eq:T}, \eqref{eq:tau}, \eqref{eq:taue}
  over $S^n$.
\end{lemma}

\begin{proof}
  The travel time $T$ is convex by
  Proposition~\ref{prop:Convexity}. By symmetry, its minimum is
  attained at a point $(\theta,\theta)$ and if $\theta \in (0, 1/2)$,
  then this point satisfies $\frac{d}{d\theta} T
  (\theta,\theta)=0$. Straightforward we find
  \begin{align*}
    T (\theta, \theta) & = 2(1-\theta) \, \tau_a (1-\theta) + 2\theta
    \, \tau_b (\theta) + (1-2\theta) \tilde \tau_e \, ,
    \\
    \frac{d}{d\theta} T (\theta,\theta) & = 2 \left( - \tau_a
      (1-\theta) - (1-\theta) \tau_a' (1-\theta) + \tau_b (\theta) +
      \theta \, \tau_b' (\theta) + \tilde\tau \right),
    \\
    \frac{d^2}{d\theta^2} T (\theta,\theta) & = 2\left( 2 \tau_a'
      (1-\theta) + (1-\theta) \tau_a'' (1-\theta) + 2 \tau_b' (\theta)
      + \theta \tau_b'' (\theta) \right),
  \end{align*}
  hence $\frac{d^2}{d\theta^2} T (\theta,\theta) > 0$, which shows
  that the map $\theta \mapsto T (\theta,\theta)$ is strictly
  convex. Hence it admits a unique point of minimum $\Theta
  (\tilde\tau)$ in $(0, 1/2)$. The standard Implicit Function Theorem
  ensures that $\Theta$ is continuous.
\end{proof}

\begin{lemma}
  \label{lem:TT}
  Let the travel time $\tau_a,\tau_b \in \C0 ([0,1]; \reali^+)$ be non
  decreasing and convex, at least one of the two being strictly
  convex. Then, there exists a map $\tilde T \in \C0 ([0, 1/2];
  \reali^+)$ such that assigning the travel time $\tilde T (\theta)$
  on road $e$ makes the configuration $(\theta,\theta)$ the unique
  local Nash point in the sense of Definition~\ref{def:Nash}.
\end{lemma}

\begin{proof}
  Given $\theta \in [0,1/2]$, we seek a $\tilde \tau$ such that
  $(\theta,\theta)$ is an equilibrium point. To this aim, we solve
  \begin{displaymath}
    \tau_a (\theta,\theta) = \tau_b (\theta,\theta)
    \qquad \qquad
    \tau_a (\theta,\theta) = \tau_\gamma (\theta,\theta) \,.
  \end{displaymath}
  By symmetry consideration, to former equality is certainly satisfied
  for any $\theta \in [0, 1/2]$. The latter is equivalent to:
  \begin{displaymath}
    \tau_a (1-\theta) + \tau_b (\theta)
    =
    2 \tau_a (1-\theta) + \tilde \tau \,.
  \end{displaymath}
  Therefore, we set
  \begin{displaymath}
    \tilde T (\theta)
    =
    \begin{cases}
      \tau_b (\theta) - \tau_a (1-\theta) & \text{if $\tau_b (\theta)
        \geq \tau_a (1-\theta)$},
      \\
      0 & \text{if $\tau_b (\theta) < \tau_a (1-\theta)$}.
    \end{cases}
  \end{displaymath}
  By construction, $(\theta,\theta)$ is an equilibrium configuration
  in the sense of Definition~\ref{def:Eq}, once the travel time
  $\tilde \tau$ along the road $e$ is set equal end $\tilde T
  (\theta)$.

  When $\theta \in (0, 1/2)$, to prove that $(\theta,\theta)$ is a
  local Nash point, thanks to the present symmetries, it is sufficient
  to check that for all small $\epsilon>0$ we have
  \begin{align*}
    \tau_\alpha (\theta+\epsilon, \theta) & > \tau_\gamma
    (\theta,\theta),
    \\
    \tau_\alpha (\theta+\epsilon, \theta-\epsilon) & > \tau_\beta
    (\theta,\theta),
    \\
    \tau_\gamma (\theta-\epsilon, \theta) & > \tau_\alpha
    (\theta,\theta),
  \end{align*}
  or, equivalently,
  \begin{align*}
    \tau_b (\theta+\epsilon) - \tau_b (\theta) + \tau_a (1-\theta) -
    \tau_a(1-\theta-\epsilon) & > 0,
    \\
    \tau_a (1-\theta+\epsilon) - \tau_a (1-\theta-\epsilon) + \tau_b
    (\theta+\epsilon) - \tau_b (\theta-\epsilon) & > 0,
    \\
    \tau_a (1-\eta+\epsilon) - \tau_a (1-\theta) & > 0,
  \end{align*}
  and all these inequalities hold by the monotonicity of the travel
  times.
\end{proof}

\begin{proof}[Proof of Theorem~\ref{thm:Control}]
  Let $\Theta$ and $\tilde T$ be the maps defined in
  Lemma~\ref{lem:Theta} and Lemma~\ref{lem:TT}, respectively. Define
  \begin{displaymath}
    \Upsilon \colon [0,1/2] \to [0,1/2]
    \qquad \mbox{ by } \qquad
    \Upsilon = \Theta \circ \tilde T,
  \end{displaymath}
  and call $\theta_*$ a fixed point for $\Upsilon$. By construction,
  $(\theta_*, \theta_*)$ is a local Nash point, once $\tilde \tau_* =
  \tilde T (\theta_*)$ is fixed as the travel time along road $e$.
\end{proof}




\end{document}